\documentclass[12pt]{amsart}

\usepackage{amsmath,amsthm,amsfonts,latexsym,amscd}

\usepackage[mathscr]{eucal}

\theoremstyle{plain}
            \newtheorem{theorem}{Theorem}[section]
            \newtheorem{lemma}[theorem]{Lemma}

\theoremstyle{definition}
            \newtheorem{definition}[theorem]{Definition}

\theoremstyle{remark}

\numberwithin{equation}{section}

\addtocounter{section}{0}

\newcommand{\script}[1]{{\mathcal{#1}}}

\newcommand{\spp}{\script P}

\newcommand{\bthrm}{\begin{theorem}}
 
\newcommand{\blem}{\begin{lemma}}
\newcommand{\elem}{\end{lemma}}
\newcommand{\bprf}{\begin{proof}}
\newcommand{\eprf}{\end{proof}}

\begin{document}

\title{Weak Paveability and the Kadison-Singer Problem}{ \thanks{This work was partially supported by T\"{U}BITAK Grant 107T896}}

\author{Charles A. Akemann, Joel Anderson and Bet\"ul Tanbay}{ \thanks{The authors wish to thank the Istanbul Center for Mathematical
Sciences for their hospitality.}}

\maketitle

\section{Introduction}

The Kadison-Singer Problem (hereinafter K-S) began with a problem in \cite{KS} and has since expanded to a very large number of equivalent problems in various fields 
(see \cite{pc} for an extensive discussion). 
In the present paper we will introduce the notion of  {\bf weak paveability} for positive elements of a von Neumann algebra $M$.  This new formulation  implies the traditional version of paveability \cite{ja1} iff K-S is affirmed ( see definitions below).  We show that the set of weakly paveable positive elements of $M^+$ is open and  norm dense in $M^+$.  Finally, we show that to affirm K-S it suffices to show that projections with compact diagonal are weakly paveable.  Therefore weakly paveable matrices will either contain a counterexample, or else weak paveability must be an easier route to affirming K-S.

\section{Definitions}

\begin{definition}  Let $M_n$ denote the $n\times n$ complex matrices and let $D_n$ denote the diagonal matrices in $M_n$.  The only matrix norm used will be the operator norm.    Let  $M = \sum_{n=2}^\infty \oplus_\infty  M_n$, so that  the 
elements of $M$ are bounded sequences $a=\{a_n\} $ , with $a_n \in M_n$.  Let $x_n$ be the unit of $M_n$ viewed as a summand of $M$. Let $K$ denote the compact elements of $M$, namely those $a \in M$ such that $\lim_{n \to \infty} ||x_n a|| = 0$. Let $\pi : M \to M/K$ be the quotient (aka Calkin) map. Recall that a state of $M$ is singular iff it vanishes on $K$. 

  It is well-known that $M$ is a von Neumann algebra under the sup norm on $M$, i.e. if $a=\{a_n\} \in M, ||a|| = \sup_n||a_n||$.  Let $D = \sum_{n=2}^\infty \oplus_\infty  D_n$, and write  $ \spp : M \to D$ for  the conditional expectation.  We shall refer to the elements of $M$ as matrices.
  
  Let $w(b)$ denote the numerical radius of the matrix $b \in M$.  
\end{definition}

\begin{definition}  A matrix $x\in M $ is {\bf paveable} if for every $\epsilon > 0$ there exist a natural number $m$ and projections $p_1, ..., p_m \in D$ such that $\sum p_j=1$ and  $||p_j(x-\spp(x))p_j|| < \epsilon$ for all $j = 1, ..., m$.
\end{definition}

One form of K-S asserts that every positive element of $M$ is paveable.  Another asserts that every pure state of $D$ extends uniquely to a state of $M$.

Remark:  Paveable matrices form a norm closed, self-adjoint subspace of $M$ \cite{bt}.  It is widely believed  that there are some matrices in $M^+$ that are not paveable.  This view is not shared by the authors.   Assume for the moment that this is the case.  Then it becomes of interest to understand which matrices are not paveable and why not.  In order to clarify this issue we introduce a  definition that looks similar, but is distinct.

\medskip

\begin{definition}  A non-compact matrix $x\in M^+$ is {\bf weakly paveable} if there exist a natural number $m$ and projections $p_1, ..., p_m \in D$ such that $\sum p_j=1$ and
$||\pi(p_j(x-\spp(x))p_j)|| < ||\pi(x)||$ for all $j = 1, ..., m$.
\end{definition}

\section{Basics of Weak Paving}

It was noted above that the set of paveable matrices in $M^+$ is closed.  By contrast, we now show  that the set of weakly paveable positive elements is open and dense in $M^+$.  From this it is clear that weak paveability implies paveability iff all matrices in $M^+$ are paveable.  (i.e.  iff K-S is affirmed.)

\begin{theorem}  The set of weakly paveable elements of $M^+$ is open for the norm topology and closed under non-zero scalar multiplication.
\end{theorem}
\begin{proof} Closure under non-zero scalar multiplication is immediate.

 Now suppose  $b \in M^+$ is  weakly paveable and $||\pi(b)||= 1$.  This means  there exist a natural number $m$ and projections $p_1, ..., p_m \in D$ with sum 1 such that 
\[
\left \|\pi\bigg(\sum_{j=1}^m p_j(b-\spp(b))p_j\bigg)\right \|=\delta <1.
\]  

Choose $c \in M^+$ such that $||\pi(b-c)||\le ||b-c||< (1/4)(1-\delta)$.  With this  we get
\noindent $||\pi(b)||-||\pi(c)|| < (1/4)(1-\delta)$, so
 \[
 1-(1/4)(1-\delta)=(3+\delta)/4 < ||\pi(c)||.
 \]
 Thus  
\begin{align*}
&\le \bigg \| \pi(\sum_{j=1}^m p_j(c-\spp(c))p_j)\bigg\| \\
&\le\bigg\| \pi(\sum_{j=1}^m p_j(c-b-\spp(c-b))p_j)\bigg\| \\
 &+ \bigg \| \pi(\sum_{j=1}^m p_j(c-\spp(c))p_j)\bigg\|\\
& \le 2\|\pi(b-c)\|+\delta \le 2 (1/4)(1-\delta) + \delta\\
&=(1+\delta)/2 <(3+\delta)/4 <||\pi(c)||,
 \end{align*}
 
so $c$ is weakly paveable.  This shows that the set of weakly paveable elements of $M^+$ is open.

\end{proof}

\begin{theorem} The set of weakly paveable elements in $M^+$ is dense in $M^+$.
\end{theorem}
\begin{proof}
Fix  $b \in M^+$  with $||\pi(b)|| = \delta$.  Let  $\epsilon > 0 $ be given and define $c = b + \epsilon 1$. Then $||\pi(c)|| =\delta + \epsilon$ but 

\begin{align*}
||\pi(c-\spp(c))|| &= ||\pi(b-\spp(b))||  \\ 
&\le ||\pi(b)||\le \delta \\
&< ||\pi(c)||. 
\end{align*}

 Thus $c$ is weakly paveable and $||c-b|| = \epsilon$.
\end{proof}

\section{Relationship to the Kadison-Singer Problem}

\begin{theorem} If $b \in M^+, ||\pi(b)||=1$, $\limsup_{n \to \infty}||\spp(bx_n)|| < 1$
and  there is a  unitary $u \in D$ such that 
\[
\limsup_{n \to \infty} w(x_nbub) =t < 1,
\]
then $b$ is weakly paveable.
\end{theorem}
\begin{proof}  
Select orthogonal spectral projections $p_1, ..., p_m$ of $u$  corresponding to arcs on the unit circle with
centers $z_1, ..., z_m$ so that $w=\sum_1^mz_jp_j$ is unitary, and $\|u-w\|<(1-t)/4$. 

If we assume that $b$ is not weakly paved by $p_1, ..., p_m$, then renumbering if necessary, we may assume  that
$||\pi(p_1(b-\spp(b))p_1)||=1$.
This means 
there are a subsequence $\{n_j\}$ of the natural numbers and vector states $\{f_{n_j}=x_{n_j}f_{n_j}\}$  such that 

$$\lim_{j \to \infty} ||x_{n_j}p_1(b-\spp(b))p_1|| = \lim_{j \to \infty} |f_{n_j}(p_1(b-\spp(b))p_1)| = 1.$$

  Since 
  $\limsup_{n \to \infty}||\spp(bx_n)|| < 1 $, then the last equality implies that $$\lim_{j \to \infty} f_{n_j}(p_1bp_1) = 1,$$ and so  there is a singular state $f$ of $M$   ($f$ can be any subnet limit of $\{f_{n_j}\}$)
such that $f(p_1bp_1)=1$ and $f(\spp(b)) = 0$.  This implies that  $p_1fp_1=f$ and that $bfb=f$ by the
Cauchy-Schwarz inequality.  But then 
$$t = \limsup_{n \to \infty} w(x_nbub) \ge \limsup_{j \to \infty} |f_{n_j}(bub)|\ge|f(bub)|=|f(u)|$$
$$\ge|f(w)|-|f(u-w)| \ge|z_1|- (1-t)/4=1-(1-t)/4 > t. $$
a contradiction.
\end{proof}

\begin{theorem} If  $q \in M^+$ is  non-compact with compact diagonal, then the following conditions are
equivalent.

\begin{enumerate}
\item   There is a  unitary $u \in D$ such that 
$$\limsup_{n \to \infty} w(x_nquq) < 1.$$

\item $q$ is weakly paveable. 

\end{enumerate}

\end{theorem}

\begin{proof} By the last theorem, (1) implies (2).

Now assume that  statement $(2)$ holds.  Because $q$ has compact diagonal,  there is a natural number $m$ and orthogonal projections $p_1, ..., p_m \in D$ such that 
\[
\sum p_j = 1 \quad\text{and} \quad||\pi(p_jqp_j)|| < t < 1\text{ for all}\quad j = 1,..., m.
\]
  Define a unitary matrix $u=\sum_1^m \lambda_jp_j$ where the $\{\lambda_j\}$ are distinct $m^{th}$ roots of unity. 

	Now assume that statement  (1)  fails for $u$ as defined above.  Then there are a subsequence $\{n_j\}$ and states $\{f_{n_j}\}$ of $M$ such that $f_{n_j}(x_{n_j}) = 1$ for all $j$ and
\[
|f_{n_j}(x_{n_j}quq)|=|\sum_{i=1}^m\lambda_if_{n_j}(x_{n_j}qp_iq)|= |\sum_{i=1}^m\lambda_i(qf_{n_j}q)(p_i)| \rightarrow_{j \to \infty} 1.
\]
Let $f$ be a limit point of the $\{f_{n_j}\}$.  Then $f$ is a singular state and 
\[
|\sum_{i=1}^m\lambda_i(qfq)(p_i)| = 1
\] 
Since the projections  $\{p_i\}$ sum to 1, this means that
$f(qp_iq) = 1$ for precisely one index $i$ (WLOG $i= 1$) and is zero for all others.  Thus $f(q) = 1$ also $f(p_1)=1$, so we get 
\[
1 > \|\pi(p_1qp_1)\| \ge  f(p_1qp_1) = 1, 
\]
a contradiction.  Thus statement (2) implies statement (1).

\end{proof}

\begin{theorem}  The following statements are equivalent.

\begin{enumerate}

\item Every non-compact  projection with compact diagonal is weakly paveable.

\item Every non-compact projection in $M$ is weakly paveable.

\item Pure states of $D$ have unique pure state extensions to $M$.
\end{enumerate}
\end{theorem}

\begin{proof} By \cite{ja1}, (3) implies (2), and clearly (2) implies (1).

 If (2) is true and (3) is false, by  \cite{ja1} there are singular pure states $f \perp g$ of $M$ that restrict to the same pure state of $D$ and $f=f\circ \spp$.  By the non-commutative Urysohn's lemma, there is a projection $q \in M$ such that $f(q)=0, g(q)=1$. 

Since (2) is assumed to be true, $q$ is weakly paveable, so there are a natural number $m$ and projections $p_1, ..., p_m \in D$ with sum 1 such that 
 $||\pi(p_j(q-\spp(q))p_j)|| < 1$ for all $j=1, ..., m$.    Since $g|_D$ is multiplicative, there is a unique $j_0$ such that $f(p_{j_0}) = 1=g(p_{j_0})$.  Now use the assumed facts about $f,g,p_{j_0}$ (such as $g|_D=f|_D$) to get: 
$$g(p_{j_0}(q-\spp(q))p_{j_0})=g(q)-g(\spp(q))=1-f(\spp(q))=1-f(q) = 1,$$
contradicting $$||\pi(p_j(q-\spp(q))p_j)|| < 1$$ for all $j=1, ..., m$. Thus statements (2) and (3) are equivalent.

Now assume 1 and we prove 2.  
First we prove an intermediate fact.

\noindent $(**)$  There exists constants $ 0 < \gamma, 0 < \epsilon \le 1/2$ such that for every finite rank projection $q \in M$  with $||\spp(q)|| < \gamma$, there is a unitary $u \in D$ such that $w(quq) < 1-\epsilon$.

If  $(**)$  is false, then we may recursively construct finite rank projections $\{q_j\} \subset M$   that have orthogonal central covers and that satisfy $||\spp(q_j)|| < 1/j$ and $w(q_juq_j) > 1-1/j$ for all unitary $u \in D$.  Define $q= \sum q_j$.  Then $q$ has compact diagonal and for all unitary $u \in D, w(quq) = 1$.  However, by assumption $q$ is weakly paveable, so by Theorem 4.2, we know that $w(quq)$ is eventually less than 1 for some unitary $u \in D$. Contradiction.  So $(**)$ is true.

Using the $\gamma, \epsilon$ from $(**)$, we claim the following is true.

\noindent$(***)$  For any projection $q \in M$ with $||\spp(q)|| < \gamma$, there exists unitary $u \in D$ such that $w(quq) \le 1- \epsilon$.

Choose a projection $q \in M$ with $||\spp(q)|| < \gamma$.  By (**)  for each $n$ we may choose unitary $u_n = x_nu_n \in D$ such that $w(qu_nq) < 1-\epsilon$.  Let $u = \sum u_n$, then 
$w(quq) \le 1-\epsilon$. This proves (***).

Therefore $q$  is weakly paveable by Theorem 4.1. 

\end{proof}

\end{document}